\documentclass[11pt,reqno]{amsart}
\usepackage{amsmath, amsfonts, amsthm, amssymb, stmaryrd, color, cite}
\usepackage{parskip}
\usepackage{graphicx}
\usepackage{latexsym,hyperref}
\usepackage{cancel}

\usepackage[dvipsnames]{xcolor}

\allowdisplaybreaks

\numberwithin{equation}{section}
\newtheorem{theorem}{Theorem}
\newtheorem{definition}{Definition}
\newtheorem{lemma}{Lemma}

\def\ve{\varepsilon}

\title[MDP for Stochastic Schr\"odinger Equation]
{{The Law of Iterated Logarithm for a Linear Stochastic Schr\"odinger Equation}}

\author[P. Fatheddin]{Parisa Fatheddin}
\address{Department of Mathematics, Ohio State University, USA.}
\email{fatheddin.1@osu.edu}


\begin{document}
\begin{abstract}
The moderate deviation principle is achieved for a stochastic Schr\"odinger type equation by applying the classical Azencott method. The Friedlin-Wentzell inequality derived by this method is then used to prove the Strassen's compact law of the iterated logarithm for the equation.
\end{abstract}
\maketitle

\section{Introduction}
Many phenomena in the field of optics may be modeled by Schr\"odinger type equations. They are especially used in the study of lasers. For example, a stochastic Schr\"odinger equation is often studied to examine the effect of noise such as air temperature on laser propagation in random media. For this model and more information on the mathematical study of lasers we refer the reader to \cite{Andrews}, chapter 6 of \cite{Fatheddin1}, and chapter 3 of \cite{(39)}. In this paper we study some small noise asymptotics for the following general form of stochastic linear Schr\"odinger equation,
\begin{eqnarray}\label{originalequation}
du^{\varepsilon}(t)&=& i \Delta u^{\varepsilon}(t)dt +\mathcal{U}(t)u^{\varepsilon}(t)dt + \sqrt{\varepsilon} g(t,u^{\varepsilon}(t))dW(t),\\
u^{\varepsilon}(0) &=& \gamma \in V,\nonumber
\end{eqnarray}
where $0<\varepsilon<1$, $V:= H_{0}^{1}(G; \mathcal{C})$ for a bounded domain with smooth boundary $G$ in $\mathbb{R}^{d}$ for $1\leq d <\infty$, $W$ is a Q-Wiener process and $\mathcal{U}$ is a deterministic complex-valued function with properties given later in Section 2. Due to the similarity in structure, we note that the results also hold for cylindrical Wiener process. Letting $H:= L^{2}(G;\mathcal{C})$, the well-posedness of variational solutions to \eqref{originalequation} is implied by Theorem 2 of \cite{Lisei1} in space $L^{2}(\Omega; \mathcal{C}([0,T];H))\cap L^{2}(\Omega \times [0,T]; V)$, where in \cite{Lisei1} the nonlinear equation was considered. Moreover, the moderate deviation principle (MDP) for \eqref{originalequation} was established in \cite{Fatheddin2} in space $\mathcal{C}([0,T];H)$, where $H:= L^{2}(G;\mathcal{C})$ by the weak convergence approach. Here we concentrate on MDP by the Azencott method in $\mathcal{C}([0,T];H)$, which enables us to achieve the Strassen's compact law of the iterated logarithm (LIL) in the same space. In the study of LIL, Strassen's compact LIL is used in the setting of stochastic partial differential equations (SPDEs). For an introduction to other types of LIL such as Khintchine, Chover and Chung LIL, we recommend \cite{(6)}. \\
\indent The connection between large or moderate deviations and the Strassen's compact LIL was first observed by J. Deuschel and D. Stroock in Lemma 1.4.1 of \cite{Deuschel}. Afterwards, following the use of a family of contractions introduced in \cite{(3)}, authors such as those in \cite{Ouahra, Caramellino, Eddahbi} have proved the compact Strassen's LIL for their SPDEs as an application of their large deviation principle (LDP) or MDP result. The family of contractions were used to prove that the process formed for LIL is relatively compact. In this paper, we follow a more direct approach to achieve this property by noting a compact embedding of the space used. The Freidlin-Wentzell inequality, which is a key step in establishing the LDP or MDP by the Azencott method, is used in Strassen's compact LIL to verify that the set introduced in \cite{Deuschel} is the limit set of the process. The rate function from the LDP or MDP result is used to define this limit set forming the main connection to Strassen's compact LIL. For more results applying this form of connection between the two theorems we refer the reader to \cite{(23), (22), Ouahra, Wu}. \\
\indent The MDP is a special case of the LDP, where for MDP the considered process converges to its deterministic counterpart at a slower rate than that used for the LDP. The weak convergence approach, first introduced in \cite{(11), BDM}, was applied in \cite{(24)} to prove the LDP for a nonlinear stochastic Schr\"odinger equation studied in \cite{Lisei1} and as mentioned earlier, it was also applied in \cite{Fatheddin2} to establish the MDP for equation (1.1) above. The more classical approach, the Azencott method, was first introduced by \cite{Azencott, Priouret}. In the literature, regarding the stochastic Schr\"odinger equation, E. Gautier achieved the LDP by Azencott method for nonlinear stochastic Schr\"odinger with additive noise in \cite{22}, linear multiplicative noise in \cite{21} and fractional noise in \cite{23}. We are not aware of a result in the literature on the LIL on any type of Sch\"odinger equation. Also the novelty of the result in this paper is that it does not require the integrand of the stochastic integral to be bounded. For key estimates needed to prove the Azencott method, most authors such as those in \cite{Chow, Peszat, Eddahbi1, Nualart, Cerrai} have relied on this bounded property. We have the boundedness of the expectation of this integrand and have used different techniques than those applied in the literature to overcome this barrier.\\
\indent We will begin in Section 2 by introducing conditions and notations used throughout the paper along with stating the main results. Section 3 is devoted to the proof of the MDP by the Azencott method and is followed by the proof of the Strassen's compact LIL in Section 4. Furthermore, we offer an in depth comparison between the weak convergence approach used in \cite{Fatheddin2} and the Azencott method applied here in Section 5.

\section{Preliminaries}
\indent To achieve the MDP for $u^{\varepsilon}(t)$ given in \eqref{originalequation}, one forms the centered process, $u^{\varepsilon}(t)-u^{0}(t)$ and proves large deviations for this centered process multiplied by a rate of decay slower than that for large deviations. More precisely, we form,
\begin{equation}\label{(2.1)}
v^{\varepsilon}(t) = \frac{a(\varepsilon)}{\sqrt{\varepsilon}}(u^{\varepsilon}(t)-u^{0}(t)), \hspace{.25cm}\text{with} \hspace{.25cm} 0<a(\varepsilon)\rightarrow 0, \hspace{.25cm} \text{and} \hspace{.25cm} \frac{a(\varepsilon)}{\sqrt{\varepsilon}}\rightarrow \infty \hspace{.25cm} \text{as} \hspace{.25cm} \varepsilon \rightarrow 0,
\end{equation}
where $\sqrt{\varepsilon}$ is the rate used for large deviations. That is,
\begin{equation}\label{MDPequation}
v^{\varepsilon}(t) = -i \hspace{.08cm}\int_{0}^{t}  Av^{\varepsilon}(s) ds + \int_{0}^{t} \mathcal{U}(s)v^{\varepsilon}(s) ds + a(\varepsilon) \int_{0}^{t} \hat{g}(s,v^{\varepsilon}(s))dW(s),
\end{equation}
with
\begin{equation}
\hat{g}(t,v^{\varepsilon}(t)):= g\left(t, \frac{\sqrt{\varepsilon}}{a(\varepsilon)} v^{\varepsilon}(t) +  u^{0}(t)\right),\label{MDPg}
\end{equation}
where $A:= -\Delta : V\rightarrow V^{*}$, with $V^{*}$ denoting the dual space of $V$ and we have for $u,v\in V$,
\begin{equation} \label{eqA}
\langle Au,v\rangle = \int_G \nabla u(x)\nabla\overline{v}(x)dx, \mbox{ for each },
\end{equation}
where $\bar{u}$ is the complex conjugate of $u$. Recall that $H:= L^{2}(G;\mathcal{C})$ and $V:= H_{0}^{1}(G; \mathcal{C})$. We denote their respective norms as $\|\cdot\|$ and $\|\cdot\|_{V}$. The linear potential, $\mathcal{U}:[0,T]\times G\to \mathbb{C}$ is measurable and partially differentiable with respect to the spatial variable and is assumed to satisfy,
\begin{equation} \label{K5}
|\mathcal{U}(t,x)|^{2} + \sum_{j=1}^{n} \left|\frac{\partial \mathcal{U}(t,x)}{\partial x_{j}}\right|^{2} \leq k_{0}, \hspace{.7cm} \text{for all } t\in [0,T], x\in G,
\end{equation}
where, throughout the paper we let $k_{i}$ for $i\in \mathbb{N}$ to denote a bounded and positive constant. Furthermore, $W$ is a $Q$-Wiener process written as a sum of independent one-dimensional real Brownian motions, $\beta_{j}$, as $W(t)= \sum_{j}^{\infty} \sqrt{\lambda_{j}}\beta_{j}(t) e_{j}$, where for each $j\in \mathbb{N}$, $\lambda_{j}$ is an eigenvalue satisfying, $Qe_{j}= \lambda_{j}e_{j}$, for a complete orthonormal system, $\{e_{j}\}_{j\geq 1}$ in $H$ and $Q$ is the covariance operator of the $Q$-Wiener process. We also let $H_{0}:= Q^{1/2}H$ and let its norm and inner product be denoted as $|u|_{0}$ and $(u,v)_{0}:= (Q^{1/2}u, Q^{1/2}v)$, respectively for $u,v\in H_{0}$. To match the setting in \cite{Fatheddin2}, we impose the following conditions on $g: [0,T]\times H\rightarrow L_{2}(H_{0},H)$ in \eqref{originalequation} for Hilbert Schmidt spaces $L_{2}(H_{0},H)$ and $L_{2}(H_{0},V)$.
\begin{align}
\|g(t,u)\|_{L_{2}(H_{0},H)}^{2} &\leq  k_{1}\hspace{.08cm}(1+\|u\|^{2}), \hspace{.3cm} \text{for all } \hspace{.3cm} u\in H, \label{(2.6)}\\
\|g(t,u)\|^{2}_{L_{2}(H_{0},V)} &\leq   k_{2}\hspace{.08cm}(1+\|u\|_{V}^{2}), \hspace{.3cm} \text{for all} \hspace{.3cm} u\in V,\label{(2.7)}\\
\|g(t,u)-g(t,v)\|_{L_{2}(H_{0},H)}^{2}&\leq  k_{3}\hspace{.08cm}\|u-v\|^{2}, \hspace{.3cm} \text{for all} \hspace{.3cm} u,v\in H,\label{(2.8)}\\
 {\|g(t,u)-g(t,v)\|_{L_{2}(H_{0},V)}^{2}} & \leq { k_{4}\hspace{.08cm}\|u-v\|^{2}_V, \hspace{.3cm} \text{for all} \hspace{.3cm} u,v\in V,} \label{K4}\\
\|g(t_{1},u)-g(t_{2},u)\|^{2}_{L_{2}(H_{0},H)} &\leq  k_{5}\hspace{.08cm}\|t_{1}-t_{2}\|^{2} \hspace{.3cm} \text{for all} \hspace{.3cm}u\in H, \hspace{.08cm} t_{1}, t_{2}\in [0,T].\label{Holder}
\end{align}
Then based on the results in \cite{Lisei1} and \cite{Fatheddin2} we have the following estimates for any $1\leq p <\infty$ and
\begin{flalign}
&\mathbb{E}\sup_{0\leq t\leq T} \|u^{\varepsilon}(t)\|^{2p} \leq   N_{2p}(T)\hspace{.08cm} (\varepsilon +\|\gamma\|^{2p}), \hspace{.05cm} \text{ and } \hspace{.05cm} \mathbb{E} \int_{0}^{T} \|u^{\varepsilon}(t)\|_{V}^{2}dt \leq N(T) (\varepsilon + \|\gamma\|_{V}^{2}),&\label{uHbound}\\
&\mathbb{E} \sup_{0\leq t\leq T} \|v^{\varepsilon}(t)\|^{2p}\leq a^{2}(\varepsilon)   {\bar N}_{2p}(T), \hspace{.05cm}\text{ and } \hspace{.05cm}
\mathbb{E}\sup_{0\leq t\leq T} \|v^{\varepsilon}(t)\|_{V}^{2} \leq \bar K (T). &\label{vVnorm}
\end{flalign}
See Theorem 2 in \cite{Lisei1} and Theorem 3 and the Appendix in \cite{Fatheddin2}. For $h\in L^{2}(0,T;H_{0})$ taking values in the Cameron-Martin space $\mathcal{H}_{0}$ defined as,
\begin{equation}\label{Cameron}
\mathcal{H}_{0}:= \left\{h\in L^{2}([0,T];H_0): h \text{ is absolutely continuous and } \int_{0}^{T}|\dot{h}(s)|^{2}_0ds<M \right\},
\end{equation}
where $M>0$ is a bounded constant and $\dot{h}(\cdot)$ denotes the derivative in time, we also consider the following deterministic equation, referred to as the skeleton equation,
\begin{equation}\label{Xhcontrolled}
X^{h}(t)= - i \int_{0}^{t} AX^{h}(s)ds + \int_{0}^{t} \mathcal{U}(s) X^{h}(s)ds + \int_{0}^{t} g(s,u^{0}(s))\dot{h}(s)ds,
\end{equation}
for which similar to the proof of Theorem 3 and Appendix in \cite{Fatheddin2} we may obtain,
\begin{equation}
\sup_{0\leq t\leq T} \|X^{h}(t)\|^{2p}\leq \hat{N}_{2p}(T, M), \hspace{.5cm} \text{ and }\hspace{.5cm} \sup_{0\leq t\leq T} \|X^{h}(t)\|_{V}^{2} \leq \hat{K} (T,M)\ \mbox{a.s.} \label{XhVnorm}
\end{equation}
Furthermore, we assume
\begin{equation} \label{epsi}
 0<\varepsilon, a(\varepsilon)< 1\hspace{.25cm} \mbox{ and }\hspace{.25cm} \frac{\varepsilon}{a^2(\varepsilon)}<1,
\end{equation}
and prove the following theorems.
\begin{theorem} \label{TH1}
Suppose conditions \eqref{(2.6)}-\eqref{Holder} hold and let $\varepsilon_{0}\in (0,1)$. Then the family of solutions  $\{u^{\varepsilon} \}_{ \varepsilon\in(0,\varepsilon_{0})}$ satisfies the moderate deviation principle in $\mathcal{C}([0,T];H)$ with rate function
\begin{align*}
I(v) = \left\{\begin{array}{ll}
& \displaystyle \frac{1}{2}\inf\left\{ \int_{0}^{T} |\dot{h}(s)|_0^{2}ds: h\in \mathcal{H}_0, v=\mathcal{G}\Big(\int_{0}^{\cdot}h(s)ds\Big)\right\},\\
&\infty,   \hspace{3cm} \text{otherwise},
\end{array}\right.
\end{align*}
where $ \mathcal{G}\Big(\int_{0}^{\cdot}h(s)ds\Big)$ is the unique variational solution corresponding to $h\in L^{2}([0,T];H_{0})$ to the skeleton equation \eqref{Xhcontrolled}.
\end{theorem}

\begin{theorem} \label{TH2}
For $\varepsilon_{0}\in (0, 1/10)$ and under conditions \eqref{(2.6)}-\eqref{(2.8)}, the family of solutions, $\{u^{\varepsilon}\}_{\varepsilon\in (0,\varepsilon_{0})}$ satisfies the Strassen's compact LIL with the limit set
 $$L:= \left\{v\in \mathcal{C}([0,T];H): I(v)\leq \frac{M}{2}\right\} ,$$ where $M$ is a positive, bounded constant and $I(v)$ is the rate function of the moderate deviations of $\{u^{\varepsilon}\}_{\varepsilon \in (0,\varepsilon_{0})}$.
\end{theorem}

\section{MDP by Azencott Method} \label{sec4}

\indent We begin by describing the Azencott method implemented here. Let  $(\mathcal{E}_{1},d_1)$ and $(\mathcal{E}_{2},d_2)$ be two Polish spaces.
Suppose $\{Y_{1}^{\varepsilon}\}_{\varepsilon>0}$ is a family of random variables that takes values in $\mathcal{E}_{1}$ and satisfies the large deviations principle with rate function  $\widetilde{I}:\mathcal{E}_1\to [0,\infty]$.
Consider the map $\Phi:\{g\in \mathcal{E}_1:\widetilde{I}(g)<\infty\} \rightarrow \mathcal{E}_{2}$ whose restriction to   compact sets  $\{\widetilde{I}\leq a\}_{a>0}$ is continuous in the topology of $\mathcal{E}_{1}$. For any $f\in \mathcal{E}_{2}$ define $$I(f)= \inf\{\widetilde{I}(g): g\in \mathcal{E}_1,\Phi(g)=f\}.$$

For another family, $\{Y_{2}^{\varepsilon}\}_{\varepsilon>0}$, the LDP holds with rate function $I$, if for any   $a>0$, $R>0, \rho>0$, there exist constants $\eta>0$ and $\varepsilon_{0}>0$ such that for all $0<\varepsilon\leq \varepsilon_{0}$ and $g\in \mathcal{E}_1$ with $\widetilde{I}(g) \le a$,
\begin{equation}\label{FreidlinWentzell}
P\Big(d_{2}\left(Y_{2}^{\varepsilon}, \Phi(g)\right)\ge \rho, d_{1}\left(Y_{1}^{\varepsilon}, g\right)<\eta \Big)\leq \exp\left(-\frac{R}{\varepsilon^{2}}\right).
\end{equation}
The above inequality is referred to as the Freidlin-Wentzell inequality. For examples of results in LDP using the Azencott method we refer the reader to \cite{Cardon, Chenal, (22), (23), Nualart}. We note that the proof given in this section is sufficient to verify the MDP for any choice of $a(\varepsilon)$ satisfying conditions given in \eqref{(2.1)}. Here since the Strassen's compact LIL can be shown as a consequence of the MDP, we prove the MDP in the case, $a(\varepsilon):= 1/\sqrt{2\log \log(1/\varepsilon)}$ and form the process,
\begin{equation}\label{formofZ}
Z^{\varepsilon}(t)= \frac{u^{\varepsilon}(t)-u^{0}(t)}{\sqrt{2\varepsilon \log \log \frac{1}{\varepsilon}}}.
\end{equation}
\noindent More precisely, we consider $Z^{\varepsilon}$ to be the variational solution of
\begin{equation}\label{(2.10)}
Z^{\varepsilon}(t) = -i \hspace{.08cm}\int_{0}^{t}\! AZ^{\varepsilon}(s) ds +  \int_{0}^{t}\!\mathcal{U}(s)Z^{\varepsilon}(s)ds+ \frac{1}{\sqrt{2\log \log \frac{1}{\varepsilon}}} \int_{0}^{t} \!\!\widetilde{g}(s,Z^{\varepsilon}(s))dW(s),
\end{equation}
for $t\in[0,T]$ and a.e. $\omega\in \Omega$, where
\begin{equation}
   \widetilde{g}(t,Z^{\varepsilon}(t)):= g\left(t, \sqrt{2\varepsilon \log \log \frac{1}{\varepsilon}}Z^{\varepsilon}(t) + u^{0}(t)\right).\label{(2.11)}
\end{equation}
Notice that $a(\varepsilon)=1/\sqrt{2\log \log(1/\varepsilon)}$ satisfies both conditions in \eqref{(2.1)} since it converges to zero at a slower rate than $\sqrt{\varepsilon}$ and for it to yield a real number, we require $\varepsilon < 1/10$. We further consider $\ve_0\le \frac{1}{10^{\sqrt{10}}}<1/10$ so that \eqref{epsi} holds for each $\ve\in(0,\varepsilon_{0})$.
Since \eqref{vVnorm} is true for any $a(\varepsilon)$ satisfying condition in \eqref{(2.1)} and \eqref{epsi}, we obtain by \eqref{vVnorm},
$$
\mathbb{E} \sup_{0\leq t\leq T} \| {Z}^{\varepsilon}(t)\|^{2p}\leq {\bar N}_{2p}(T)\hspace{.2cm} \mbox{ for }\hspace{.2cm}p\ge 1,\hspace{.25cm}
\mbox{ and }\hspace{.2cm}\mathbb{E}\sup_{0\leq t\leq T} \|{Z}^{\varepsilon}(t)\|_{V}^{2}  \leq  \bar K (T),$$
where we noted that $a(\varepsilon)<1$. In the setting of SPDEs, the map $\Phi$, given in the definition of the Azencott method, is the map $h\mapsto X^{h} $, with $X^{h}$ being the unique variational solution to \eqref{Xhcontrolled}. Namely,
\begin{equation}\label{controlled}
(X^{h}(t),v) = - i \hspace{.08cm}\int_{0}^{t} \langle AX^{h}(s),v\rangle ds + \int_{0}^{t} (\mathcal{U}(s)X^{h}(s),v)ds + \int_{0}^{t}( g(s,u^{0}(s))\dot{h}(s),v)ds,
\end{equation}
for all $t\in[0,T]$, $v\in V$. Also we know by an extended version of the Schilder's theorem that $Y_{1}^{\varepsilon}= (1/\sqrt{2\log \log(1/\varepsilon)})W$ satisfies the large deviation principle, where $W$ is a $Q$-Wiener process (see Theorem 3.1 in Chapter 6 of \cite{(1)}).\\
\indent We denote by $\|\cdot\|_{\mbox{\scalebox{0.5}{$\infty$}}}$ the norm in the space $\mathcal{C}([0,T];H)$    and by $ \|\cdot \|_{\mbox{\scalebox{0.5}{$\infty,0$}}}$  the norm in $\mathcal{C}([0,T];H_0)$ and devote the rest of the section to proving the conditions required by the Azencott method. For better presentation, we let $C_{i}$ for $i\in \mathbb{N}$ represent bounded positive constants.

\begin{lemma}
The map $h\mapsto X^{h}$ is continuous in the uniform topology when restricted on the set $\{\|h\|_{\mathcal{H}_{0}} \leq a\}$ for any fixed $a>0$.
\end{lemma}
 \begin{proof}
 \indent Let $a>0$ and choose $\|h\|_{\mathcal{H}_{0}}\vee \|k\|_{\mathcal{H}_{0}}\leq a$. Using equation \eqref{Xhcontrolled} we obtain,
\begin{flalign*}
 \sup_{0\leq s\leq t} \|X^{h}(s)-X^{k}(s)\|^{2} &\leq 2 \int_{0}^{t}\text{Im} \|X^{h}(s)-X^{k}(s)\|_{V}^{2}ds\\&+2   \int_{0}^{t}\left|\left(\mathcal{U}(s)\left(X^{h}(s)-X^{k}(s)\right), X^{h}(s)-X^{k}(s)\right)\right|ds&\\
& + 2 \int_{0}^{t} \left|\left(g\left(s,u^{0}(s)\right)\left(\dot{h}(s)-\dot{k}(s)\right), X^{h}(s)-X^{k}(s)\right)\right|ds,&
\end{flalign*}
which by applications of H\"older's and Young's inequalities, using \eqref{(2.6)}, \eqref{K5}, and the assumption $u^{0}(t)\in V$, leads to
\begin{eqnarray}
&&\sup_{0\leq s\leq t}\|X^{h}(s)-X^{k}(s)\|^{2} \leq  2 \sqrt{k_{5}} \int_{0}^{t}  \|X^{h}(s)-X^{k}(s)\|^{2}ds\label{difference}\\
&&+2 \int_{0}^{t} \|g(s,u^{0}(s))\|_{L_{2}}\hspace{.08cm} |\dot{h}(s)-\dot{k}(s)|\hspace{.08cm} \|X^{h}(s)-X^{k}(s)\|ds\nonumber\\
&\leq& 2 \sqrt{k_{5}} \int_{0}^{t}  \|X^{h}(s)-X^{k}(s)\|^{2}ds + 2C_{1}\int_{0}^{t} \|X^{h}(s)-X^{k}(s)\|\hspace{.08cm}|\dot{h}(s)-\dot{k}(s)|ds.\nonumber
\end{eqnarray}
Similar to the proof in \cite{Cardon, Chenal}, we use time discretization to estimate the second integral above by a term involving $\|h-k\|_{\infty}$. For {$n\in \mathbb{N}$}  denote $t_{i}^{n}:= iT2^{-n}$ for each $i\in\{0,1,2,..,2^{n}-1\}$ and
consider the step function $\hat\psi_n:[0,T]\to[0,T]$ defined by
\begin{equation} \label{psin2}
\hat\psi_n(s)=t_i^n,  \mbox{ if } s\in \left[ t_i^n, t_{i+1}^n\right),\  i\in\{0,1,...,2^{n}-1\} \mbox{ and }\hat\psi_n(T)=T,
\end{equation}
then using the first upper endpoint, $\frac{T}{2^{n}}$, in the discretization, we can write,
\begin{flalign}
&2C_{1} \int_{0}^{t} \|X^{h}(s)-X^{k}(s)\| \hspace{.08cm}|\dot{h}(s)-\dot{k}(s)| ds \label{sum}\\
&= 2C_{1} \int_{0}^{t} \|X^{h}(s)-X^{k}(s)\| \hspace{.08cm}|\dot{h}(s)-\dot{k}(s)|\hspace{.08cm} 1_{\{t< \frac{T}{2^{n}}\}} ds \nonumber\\
&+ 2C_{1} \int_{0}^{t} \|X^{h}(s)-X^{k}(s)\|\hspace{.08cm} |\dot{h}(s)-\dot{k}(s)|\hspace{.08cm} 1_{\{t\geq \frac{T}{2^{n}}\}} ds\nonumber \\
&\leq 2C_{1} \left\|X^{h}\left(\frac{T}{2^n}\right)- X^{k} \left(\frac{T}{2^n}\right)\right\| \hspace{.08cm}|h(s)-k(s)|\hspace{.08cm} [0,t]\nonumber\\
& + 2C_{1} \sum_{i=0}^{2^{n}-2} \left\| X^{h}\left(t- \frac{i}{2^{n}}T\right) - X^{k}\left(t- \frac{i}{2^{n}}T\right)\right\| \hspace{.04cm} |h(s)-k(s)| \hspace{.04cm}\left[ \frac{(i+1)T}{2^{n}} \wedge t, \frac{(i+2)T}{2^n} \wedge t\right]\nonumber\\
&\leq K(n, T, \hat{N}_{2}(T,M))\hspace{.08cm} \|h(s)-k(s)\|_{\infty}^{2}, \nonumber
\end{flalign}
where the second integral above was written as a sum of the integrand over each time interval from $\frac{T}{2^{n}}$ up to time $t$. We also used the fact that $\|X^{h}\left(\frac{T}{2^n}\right)- X^{k} \left(\frac{T}{2^n}\right)\|$ and $\| X^{h}\left(t- \frac{i}{2^{n}}T\right) - X^{k}\left(t- \frac{i}{2^{n}}T\right)\|$ are constants and applied \eqref{XhVnorm}. Thus, by \eqref{difference} and Gronwall's inequality we obtain the result.
\end{proof}
Henceforth, we restrict $\varepsilon_{0}\in (0,\ve_0^*)$, where $\ve_0^*= \frac{1}{10^{\sqrt{10}}}$ to achieve the LIL, however, the general MDP holds true for any $\varepsilon\in (0,1)$.
\begin{theorem}
Assume \eqref{(2.6)}-\eqref{(2.8)}, then for all  $h\in \mathcal{H}_{0}$ and constants  $R>0, \rho>0$, there exists a constant $\eta>0$ such that for every $\varepsilon \in (0,\varepsilon_{0})$
\begin{equation}\label{(3.2)}
P\left(\|Z^{\varepsilon} -X^{h} \|_{\mbox{\scalebox{0.5}{$\infty$}}}\ge \rho, \Bigg\| \dfrac{1}{\sqrt{2\log \log \frac{1}{\varepsilon}}} W-h\Bigg\|_{\mbox{\scalebox{0.5}{$\infty,0$}}} <\eta\right) \leq \exp\left(-2R\log \log \frac{1}{\varepsilon}\right).
\end{equation}
\end{theorem}
\begin{proof}
\indent It may be shown by an application of the Girsanov transformation theorem that inequality \eqref{(3.2)} is implied by the estimate below (see for instance the proof given in \cite{Nualart}),
\begin{equation}\label{(3.3)}
P\left(\|\widetilde{Z}^{\varepsilon} -X^{h} \|_{\mbox{\scalebox{0.5}{$\infty$}}}\ge \rho, \Bigg\| \frac{1}{\sqrt{2\log \log \frac{1}{\varepsilon}}}
W\Bigg\|_{\mbox{\scalebox{0.5}{$\infty,0$}}} <\eta\right) \leq \exp\left(-2R\log \log \frac{1}{\varepsilon}\right),
\end{equation}
where $\widetilde{Z}^{\varepsilon}$ is the variational solution of
\begin{align}\label{(3.4)}
\widetilde{Z}^{\varepsilon}(t) &=
 - i \hspace{.08cm} \int_{0}^{t} A\widetilde{Z}^{\varepsilon}(s) ds
+ \int_{0}^{t} \mathcal{U}(s)\widetilde{Z}^{\varepsilon}(s)ds
  + \frac{1}{\sqrt{2\log \log \frac{1}{\varepsilon}}}
\int_{0}^{t} \widetilde{g}(s,\widetilde{Z}^{\varepsilon}(s))dW(s) \nonumber\\ &\hspace{2cm}+ \int_{0}^{t} \widetilde{g}(s,\widetilde{Z}^{\varepsilon}(s))\dot{h}(s)ds,\hspace{1cm} t\in[0,T], \mbox{ a.e. }\omega\in \Omega.
\end{align}
Similar to the proof of \eqref{vVnorm} we obtain
\begin{equation}\label{Zk}
\mathbb{E}\sup_{0\leq t\leq T} \|\widetilde{Z}^{\varepsilon}(t)\|^{2p} \leq {N}_{2p}(T, M), \hspace{.25cm} \mbox{ and } \hspace{.25cm}
  \mathbb{E}\sup_{0\leq t\leq T} \|\widetilde{Z}^{\varepsilon}(t)\|_{V}^{2}  \leq  {K}(T,M),
\end{equation}
Now to establish \eqref{(3.3)}, we use the time discretization introduced in Lemma 1 to prove that for $h\in \mathcal{H}_{0}$ and any constants $R>0, \rho>0$ and $0<\beta<\rho$, there exists a constant $\eta>0$ and $\varepsilon_{0}\in (0,\ve_0^*)$, such that for all $\varepsilon\in (0,\varepsilon_{0})$ the following two inequalities hold,
\begin{equation}\label{(3.6)}
P\left(\|\widetilde{Z}^{\varepsilon} -\widetilde{Z}^{\varepsilon}(\hat\psi_n(\cdot))\|_{\mbox{\scalebox{0.5}{$\infty$}}}   >\beta\right) \leq \exp\left(-2R\log \log \frac{1}{\varepsilon}\right),
\end{equation}
and
\begin{align}\label{(3.7)}
& P\left(\|\widetilde{Z}^{\varepsilon} -X^{h} \|_{\mbox{\scalebox{0.5}{$\infty$}}} \ge \rho, \Bigg\|\frac{1}{\sqrt{2\log \log \frac{1}{\varepsilon}}}W\Bigg\|_{\mbox{\scalebox{0.5}{$\infty,0$}}} <\eta, \|\widetilde{Z}^{\varepsilon} -\widetilde{Z}^{\varepsilon}(\hat\psi_n(\cdot))\|_{\mbox{\scalebox{0.5}{$\infty$}}} \leq \beta\right)\nonumber\\
& \leq \exp\left(-2R\log \log \frac{1}{\varepsilon}\right).
\end{align}
\indent \textbf{Step 1:} Proof of inequality \eqref{(3.6)}. \\
Observe that
\begin{align}\label{sum}
& P\left(\|\widetilde{Z}^{\varepsilon} -\widetilde{Z}^{\varepsilon}(\hat\psi_n(\cdot))\|_{\mbox{\scalebox{0.5}{$\infty$}}}   >\beta\right) = P\left( \sup_{t_{i}^{n}\leq s< t_{i+1}^n}\|\widetilde{Z}^{\varepsilon}(s) - \widetilde{Z}^{\varepsilon}(\hat\psi_n(s))\|^{2}>{\beta^{2}}\right)&\nonumber\\
&\le  \sum_{i=0}^{2^n-1} P\left( \sup_{t_{i}^{n}\leq s< t_{i+1}^n}\|\widetilde{Z}^{\varepsilon}(s) - \widetilde{Z}^{\varepsilon}(\hat\psi_n(s))\|^{2}>{\beta^2}\right) \,.&
\end{align}
We apply the It\^o's formula to \eqref{(3.4)} and then take the supremum on time, to arrive at,
\begin{align*}
&\sup_{t_{i}^{n}\leq s<t_{i+1}^{n}}\|\widetilde{Z}^{\varepsilon}(s) - \widetilde{Z}^{\varepsilon}(\hat\psi_n(s))\|^{2} \leq 2     \int_{t_{i}^{n}}^{t_{i+1}^{n}}\Big| \langle A\widetilde{Z}^{\varepsilon}(s), \widetilde{Z}^{\varepsilon}(s)- \widetilde{Z}^{\varepsilon}(\hat\psi_n(s))\rangle\Big| ds \\
&+ 2  \int_{t_{i}^{n}}^{t_{i+1}^{n}}\left|  (\mathcal{U}(s)\widetilde{Z}^{\varepsilon}(s), \widetilde{Z}^{\varepsilon}(s)-\widetilde{Z}^{\varepsilon}(\hat\psi_n(s)) )\right|ds&\\
&+ \frac{2}{\sqrt{2\log \log \frac{1}{\varepsilon}}}  \sup_{t_{i}^{n} \leq s<t_{i+1}^{n}} \left|\int_{t_{i}^{n}}^{s} \text{Re} (\widetilde{g}(r, \widetilde{Z}^{\varepsilon}(r))dW(r), \widetilde{Z}^{\varepsilon}(r)-\widetilde{Z}^{\varepsilon}(\hat\psi_n(r)) )\right| \\
& + 2\hspace{.08cm} \int_{t_{i}^{n}}^{t_{i+1}^{n}}\left| (\widetilde{g}(s, \widetilde{Z}^{\varepsilon}(s))\dot{h}(s), \widetilde{Z}^{\varepsilon}(s)-\widetilde{Z}^{\varepsilon}(\hat\psi_n(s)) )\right|ds+ \frac{1}{2\log \log \frac{1}{\varepsilon}} \int_{t_{i}^{n}}^{t_{i+1}^{n}} \|\widetilde{g}(s,\widetilde{Z}^{\varepsilon}(s))\|_{L_{2}}^{2}ds&\\
&= I_{1}^{i} + I_{2}^{i}+I_{3}^{i}+I_{4}^{i}+I_{5}^{i}.
\end{align*}
By \eqref{Zk} and the Cauchy-Schwarz inequality, we have for the first term
\begin{eqnarray*}
\mathbb{E}(I_{1}^{i})  \leq    \mathbb{E} \int_{t_{i}^{n}}^{t_{i+1}^{n}} \|\widetilde{Z}^{\varepsilon}(s)\|_{V}^{2} ds +   \mathbb{E} \int_{t_{i}^{n}}^{t_{i+1}^{n}} \|\widetilde{Z}^{\varepsilon}(s)\|_{V} \|\widetilde{Z}^{\varepsilon}(\hat\psi_n(s))\|_{V}ds
\le C_{1}\hspace{.08cm} K(T,M) (t_{i+1}^{n}-t_{i}^{n}).
\end{eqnarray*}
Also with the help of \eqref{K5}, \eqref{Zk} and Young's inequality, we obtain,
\begin{eqnarray*}
\mathbb{E}(I_{2}^{i})&\leq&   k_{0} \hspace{.08cm} \mathbb{E} \int_{t_{i}^{n}}^{t_{i+1}^{n}} \|\widetilde{Z}^{\varepsilon}(s)\|^{2}ds +  \hspace{.08cm} \mathbb{E} \int_{t_{i}^{n}}^{t_{i+1}^{n}} \|\widetilde{Z}^{\varepsilon}(s) - \widetilde{Z}^{\varepsilon}(\hat\psi_n(s))\|^{2}ds\\
&\leq& C_{2}\hspace{.08cm} N_{2}(T,M) (t_{i+1}^{n}-t_{i}^{n}).
\end{eqnarray*}
Moreover, by \eqref{(2.11)}, \eqref{(2.6)} and \eqref{uHbound} ($\ve=0$) we write
\begin{align*}
\|\widetilde{g}(s,\widetilde{Z}^{\varepsilon}(s))\|^2_{L_2} \le k_1\hspace{.08cm}(1+4\varepsilon\log \log \frac{1}{\varepsilon}\|\widetilde{Z}^{\varepsilon}(s)\|^2+2\|u^0(s)\|^2)\le C_3(1+ \varepsilon\log \log \frac{1}{\varepsilon}\|\widetilde{Z}^{\varepsilon}(s)\|^2),
\end{align*}
where $C_3$ is a constant depending on $k_1, \widetilde{N}_2(T),$ and $\|\gamma\|^2.$ By the Cauchy-Schwarz and   H\"older's inequalities, we have for $h\in \mathcal{H}_{0}$,
\begin{flalign*}
&\mathbb{E}(I_{4}^{i})+ \mathbb{E}(I_{5}^{i}) \leq 2\hspace{.08cm} \mathbb{E} \int_{t_{i}^{n}}^{t_{i+1}^{n}} \|\widetilde{g}(s,\widetilde{Z}^{\varepsilon}(s))\|_{L_{2}} \hspace{.08cm} |\dot{h}(s)|   (\|\widetilde{Z}^{\varepsilon}(s)\|+\|\widetilde{Z}^{\varepsilon}(\hat\psi_n(s))\| ) ds&\\
& +    \frac{C_3}{2\log \log \frac{1}{\varepsilon}}\hspace{.08cm} \mathbb{E}\int_{t_{i}^{n}}^{t_{i+1}^{n}} \left( 1+ \varepsilon \log \log \frac{1}{\varepsilon}\|\widetilde{Z}^{\varepsilon}(s)\|^{2}\right)ds&\\
&\leq 2\hspace{.08cm}\sqrt{M} \left(\mathbb{E} \int_{t_{i}^{n}}^{t_{i+1}^{n}} \|\widetilde{g}(s,\widetilde{Z}^{\varepsilon}(s))\|_{L_{2}}^{2} (\|\widetilde{Z}^{\varepsilon}(s)\| + \|\widetilde{Z}^{\varepsilon}( \hat\psi_n(s))\|)^2ds \right)^{1/2}   &\\
&+ \left(\frac{1}{ \log \log \frac{1}{\varepsilon}}+ \varepsilon N_{2}(T, M)\right)\frac{C_3}{2}\left(t_{i+1}^{n}-t_{i}^{n}\right)&\\
&\leq 4\sqrt{M C_3} \left( N_{2}(T, M)+ \varepsilon \log \log \frac{1}{\varepsilon}N_{4}(T, M)\right)^{1/2} \left(t_{i+1}^{n}-t_{i}^{n}\right)^{1/2}&\\
& + \left(\frac{1}{2\log \log \frac{1}{\varepsilon}}+ \varepsilon  N_{2}(T, M)\right)\frac{C_3}{2} \left(t_{i+1}^{n}-t_{i}^{n}\right).&
\end{flalign*}
Note that for $\ve\in(0,\ve_0)$,  the expressions $\frac{1}{2\log \log \frac{1}{\varepsilon}}$ and $\varepsilon \log \log \frac{1}{\varepsilon}$ are bounded. Thus, we have,
\begin{eqnarray}\label{exponentialbound}
&&\sum_{i=0}^{2^{n}-1} P\left(\|\widetilde{Z}^{\varepsilon} -\widetilde{Z}^{\varepsilon}(\hat\psi_n(\cdot))\|_{\mbox{\scalebox{0.5}{$\infty$}}}   >\beta\right)\nonumber \\
&\leq& \sum_{i=0}^{2^{n}-1} P\left(I_{1}^{i} + I_{2}^{i} + I_{4}^{i} + I_{5}^{i} > \frac{\beta^{2}}{2}\right)  + \sum_{i=0}^{2^{n}-1} P\left(I_{3}^{i} > \frac{\beta^{2}}{2}\right)\nonumber\\
&\leq& \sum_{i=0}^{2^{n}-1} P\left( C(T,M) (t_{i+1}^{n} - t_{i}^{n})^{1/2} > \frac{\beta^{2}}{2}\right) + \sum_{i=0}^{2^{n}-1} P\left(I_{3}^{i} > \frac{\beta^{2}}{2}\right),
\end{eqnarray}
where noting that $t_{i+1}^{n} - t_{i}^{n} < \frac{T}{2^{n}}$, we may drop the first probability above for any $\beta>0$, by setting $n$ to become arbitrary large. To obtain the exponential bound in \eqref{(3.6)} from the probability involving the stochastic integral, $I_{3}^{i}$, many authors such as those in \cite{Chow, Peszat, Eddahbi1, Nualart, Cerrai} have made the assumption that the integrand or the quadratic variation of their stochastic integral is bounded. Here we have the expectation of the integrand is bounded. We proceed as follows implementing some ideas from the proof of Proposition 2.1 in \cite{Seidler}. For a positive number $\lambda \in \left(0, \frac{1}{2e (24)^{2}k_{1}N_{2}(T,M)}\right)$, we apply the Doob's martingale inequality to obtain
\begin{flalign}\label{stochasticterm}
&P\left(I_{3}^{i}\geq \frac{\beta^{2}}{2}\right)&\\
&\leq P\Bigg( \sup_{t_{i}^{n} \leq s<t_{i+1}^{n}} \Big|\int_{t_{i}^{n}}^{s}  (\widetilde{g}(r, \widetilde{Z}^{\varepsilon}(r))dW(r), \widetilde{Z}^{\varepsilon}(r)-\widetilde{Z}^{\varepsilon}(\hat\psi_n(r)) )\Big| \geq \frac{\beta^{2} \sqrt{2\log \log \frac{1}{\varepsilon}}}{4}\Bigg)&\nonumber\\
&\leq \exp\Bigg(-\frac{2 \beta^{4}\lambda \log \log \frac{1}{\varepsilon}}{16}\Bigg)&\nonumber\\
 & \hspace{3cm}\times \mathbb{E}\exp\left(
\sup_{t_{i}^{n} \leq s<t_{i+1}^{n}}\lambda \Big|\int_{t_{i}^{n}}^{s} (\widetilde{g}(r, \widetilde{Z}^{\varepsilon}(r))dW, \widetilde{Z}^{\varepsilon}(r)-\widetilde{Z}^{\varepsilon}(\hat\psi_n(r)) )
\Big|^{2}  \right).&\nonumber
\end{flalign}
Moreover, for $p>1$ and a continuous martingale, $(M_{t})_{t}$, using the fact that in Burkholder-Davis-Gundy inequality,
\begin{equation*}
\mathbb{E} \sup_{0\leq t\leq T} |M_{t}|^{p} \leq C_{p}\hspace{.08cm} \mathbb{E} <M>_{T}^{p/2},
\end{equation*}
the constant $C_{p}$ is $(3p^{3})^{p}$ (see Theorem 1.76 in \cite{Pardoux} for a proof), we have for $k\geq 1$ independent of $n$,
\begin{eqnarray*}\label{CT}
&&\mathbb{E} \sup_{t_{i}^{n} \leq s< t_{i+1}^{n}} \Big|\int_{t_{i}^{n}}^{s}  (\widetilde{g}(r, \widetilde{Z}^{\varepsilon}(r))dW(r), \widetilde{Z}^{\varepsilon}(r)-\widetilde{Z}^{\varepsilon}( \hat\psi_n(r) ) )\Big|^{2k} \\
&\leq& (24)^{2k} k^{6k} \hspace{.08cm}\mathbb{E}\left(\int_{t_{i}^{n}}^{t_{i+1}^{n}} \|\widetilde{g}(s,\widetilde{Z}^{\varepsilon}(s))\|_{L_{2}}^{2}
\|\widetilde{Z}^{\varepsilon}(s)-\widetilde{Z}^{\varepsilon}(\hat\psi_n(s))\|^{2}ds\right)^{k}\nonumber\\
&\leq& (24)^{2k} k^{6k} \hspace{.08cm} \mathbb{E}\left(\int_{t_{i}^{n}}^{t_{i+1}^{n}} 2k_{1}\left(1+ \ve\log \log \frac{1}{\varepsilon}\|\widetilde{Z}^{\varepsilon}(s)\|^{2} \right)\left(\|\widetilde{Z}^{\varepsilon}(s)\|^{2} + \|\widetilde{Z}^{\varepsilon}(\hat\psi_n(s))\|^{2}\right)ds\right)^{k}\nonumber\\
&\leq& (2k_{1})^{k} (24)^{2k} k^{6k} \left( N_{2}(T,M) + \left(\varepsilon \log \log \frac{1}{\varepsilon}\right)^{k} N_{4}(T,M)\right)^{k} \left(\frac{T}{2^{n}}\right)^{k},
\end{eqnarray*}
where we noted that $|t_{i+1}^{n}-t_{i}^{n}|<\frac{T}{2^{n}}$. Comparing the rates of convergence, we observe that $\varepsilon$ tends to zero faster than $1/(\log \log (1/\varepsilon))$ and $\frac{T}{2^{n}} \xrightarrow{n\rightarrow \infty} 0$ at a faster rate than $k^{5} \xrightarrow{k\rightarrow \infty} \infty$. Thus, for small enough $\varepsilon >0$ and for any $k\geq 1$, we may pick $n\geq 1$ large enough to obtain,
\begin{eqnarray*}
&&\mathbb{E} \sup_{t_{i}^{n} \leq s< t_{i+1}^{n}} \Big|\int_{t_{i}^{n}}^{s}  (\widetilde{g}(r, \widetilde{Z}^{\varepsilon}(r))dW(r), \widetilde{Z}^{\varepsilon}(r)-\widetilde{Z}^{\varepsilon}( \hat\psi_n(r) ) )\Big|^{2k}\\
&\leq& \left(2(24)^2 k_{1}N_{2}(T,M)\right)^{k} k^{k} \left(k^{5} \frac{T}{2^{n}}\right)^{k}\leq  \left(2(24)^{2} k_{1}N_{2}(T,M)\right)^{k} k^{k}.
\end{eqnarray*}
We now multiply both sides by $\lambda^{k}/k!$ and take the sum on $k$ to arrive at
\begin{eqnarray*}
&&\mathbb{E} \left(\sum_{k=0}^{\infty} \frac{\lambda^{k}}{k!} \sup_{t_{i}^{n} \leq s<t_{i+1}^{n}} \Big|\int_{t_{i}^{n}}^{s} (\widetilde{g}(r, \widetilde{Z}^{\varepsilon}(r))dW, \widetilde{Z}^{\varepsilon}(r)-\widetilde{Z}^{\varepsilon}(\hat\psi_n(r)) )
\Big|^{2k}\right) \\
&&\leq \sum_{k=0}^{\infty} \frac{\lambda^{k}}{k!} \left(2(24)^{2}k_{1}N_{2}(T,M)\right)^{k} k^{k} \leq \sum_{k=0}^{\infty} \left(2e (24)^{2}k_{1}N_{2}(T,M) \right)^{k} \lambda^{k},
\end{eqnarray*}
where we observed that $\frac{k^{k}}{k!} \leq \sum_{k=0}^{\infty} \frac{k^{k}}{k!} = e^{k}$. Then by the choice of $\lambda$ we have by means of the geometric series,
\begin{eqnarray*}
&&\mathbb{E} \left( \exp \left(\sup_{t_{i}^{n} \leq s<t_{i+1}^{n}} \lambda \Big|\int_{t_{i}^{n}}^{s}  (\widetilde{g}(r, \widetilde{Z}^{\varepsilon}(r))dW, \widetilde{Z}^{\varepsilon}(r)-\widetilde{Z}^{\varepsilon}(\hat\psi_n(r)) )
\Big|^{2}\right)\right)\\
&\leq& \frac{\lambda 2e (24)^{2}k_{1}N_{2}(T,M)}{1- \lambda 2e (24)^{2}k_{1}N_{2}(T,M)},
\end{eqnarray*}
which in \eqref{stochasticterm} yields,
\begin{equation*}
P\left(I_{3}^{i} \geq \frac{\beta^{2}}{2}\right)\leq \exp\left(-\frac{2\beta^{4} \lambda \log \log \frac{1}{\varepsilon}}{16}\right) \frac{\lambda 2e (24)^{2}k_{1}N_{2}(T,M)}{1- \lambda 2e (24)^{2}k_{1}N_{2}(T,M)},
\end{equation*}
which holds for any $0<\beta<\rho$ and $\lambda \in \left(0, \frac{1}{2e (24)^{2}k_{1}N_{2}(T,M)}\right)$. Thus, we may let $\lambda = \frac{1}{4e (24)^{2}k_{1}N_{2}(T,M)}$ and for any $R>0$, we may choose $0<\beta<\rho$ small enough to obtain \eqref{(3.6)}.\\

\textbf{Step 2:} Proof of inequality \eqref{(3.7)}. \\
Note that inequality \eqref{(3.7)} may be written as
\begin{eqnarray}\label{secondinequality2}
&&P\left(\|\widetilde{Z}^{\varepsilon}(\widetilde{\psi}_{n}(\cdot)) - X^{h}\|_{\infty} \geq \varphi, \left\|\frac{1}{\sqrt{2\log \log \frac{1}{\varepsilon}}}W\right\|_{\infty,0} <\eta, \|\widetilde{Z}^{\varepsilon}-\widetilde{Z}^{\varepsilon}(\hat{\psi}_{n}(\cdot))\|_{\infty} \leq \beta \right)\nonumber\\
&\leq& \exp\left(-2R\log \log \frac{1}{\varepsilon}\right),
\end{eqnarray}
where $\varphi:= \rho -\beta$ from the observation,
\begin{equation*}
\rho\leq \|\widetilde{Z}^{\varepsilon}-X^{h}\|_{\infty} \leq \beta + \| \widetilde{Z}^{\varepsilon}(\hat{\psi}_{n}(\cdot))- X^{h}\|_{\infty}.
\end{equation*}
Without loss of generality we assume $0<t_{i}^{n}\leq t$ and find by It$\hat{o}$'s formula,
\begin{flalign*}
&|X^{h}(t) - \widetilde{Z}^{\varepsilon}(\hat{\psi}_{n}(t))|^{2} \leq 2\hspace{.08cm} \text{Im} \int_{t_{i}^{n}}^{t}\|X^{h}(s)-\widetilde{Z}^{\varepsilon}(s)\|_{V}^{2} ds&\\
& + 2 \hspace{.08cm} \text{Re} \int_{t_{i}^{n}}^{t} \left(\mathcal{U}(s)\left(X^{h}(s)-\widetilde{Z}^{\varepsilon}(s)\right), X^{h}(s)- \widetilde{Z}^{\varepsilon}(s)\right)ds &\\
& + 2 \hspace{.08cm} \text{Re} \int_{t_{i}^{n}}^{t} \left(g(s,u^{0}(s))\dot{h}(s), X^{h}(s)- \widetilde{Z}^{\varepsilon}(s)\right)ds&\\
& + 2 \hspace{.08cm} \text{Re} \int_{0}^{t_{i}^{n}} \left(\left(g(s,u^{0}(s))- \widetilde{g}(s, \widetilde{Z}^{\varepsilon}(s))\right)\dot{h}(s), X^{h}(s)- \widetilde{Z}^{\varepsilon}(s)\right)ds&\\
&+ \frac{1}{2\log \log \frac{1}{\varepsilon}} \int_{0}^{t_{i}^{n}} \|\widetilde{g}(s, \widetilde{Z}^{\varepsilon}(s))\|_{L_{2}}^{2}ds&\\
&+ \frac{2}{\sqrt{2\log \log \frac{1}{\varepsilon}}} \int_{0}^{t_{i}^{n}} \left|\left(\widetilde{g}(s,\widetilde{Z}^{\varepsilon}(s))dW(s), X^{h}(s)- \widetilde{Z}^{\varepsilon}(s)\right)\right|&\\
&= J_{1} + J_{2} + J_{3} + J_{4} + J_{5} + J_{6}. &
\end{flalign*}
Then by introducing the terms,
\begin{eqnarray*}
J_{7}&:=&  \frac{2}{\sqrt{2\log \log \frac{1}{\varepsilon}}} \int_{0}^{t_{i}^{n}} \left(\widetilde{g}(s, \widetilde{Z}^{\varepsilon}(s))- \widetilde{g}(s, \widetilde{Z}^{\varepsilon}(s_{i}^{n}))dW(s), \widetilde{Z}^{\varepsilon}(s)-\widetilde{Z}^{\varepsilon}(s_{i}^{n})\right),\\
J_{8}&:=& \frac{2}{\sqrt{2\log \log \frac{1}{\varepsilon}}} \int_{0}^{t_{i}^{n}} \left(\widetilde{g}(s, \widetilde{Z}^{\varepsilon}(s_{i}^{n})dW(s), \widetilde{Z}^{\varepsilon}(s)-\widetilde{Z}^{\varepsilon}(s_{i}^{n}))\right),\\
J_{9}&:=& \frac{2}{\sqrt{2\log \log \frac{1}{\varepsilon}}} \int_{0}^{t_{i}^{n}} \left(\widetilde{g}(s, \widetilde{Z}^{\varepsilon}(s)dW(s), \widetilde{Z}^{\varepsilon}(s^{n}_{i})-X^{h}(s))\right),
\end{eqnarray*}
we may represent the probability in \eqref{secondinequality2} as the following sum,
\begin{flalign*}
& P\left(J_{1}+J_{2}+J_{3}+J_{4}+J_{5}\geq \frac{\varphi}{4} \right)& \\
&+ P\left(J_{7} \geq \frac{\varphi}{4}, \hspace{.08cm}\left\|\frac{1}{\sqrt{2\log \log \frac{1}{\varepsilon}}} W\right\|_{\infty, 0}\leq \eta, \hspace{.08cm} \|\widetilde{Z}^{\varepsilon}- \widetilde{Z}^{\varepsilon}(\hat{\psi}_{n}(\cdot))\|_{\infty} \leq \beta\right)&\\
&+ P\left(J_{8} \geq \frac{\varphi}{4}, \hspace{.08cm}\|\widetilde{Z}^{\varepsilon}- \widetilde{Z}^{\varepsilon}(\hat{\psi}_{n}(\cdot))\|_{\infty} \leq \beta\right)+ P\left(J_{9} \geq \frac{\varphi}{4}\right)&\\
&= P_{1}+P_{2}+P_{3}+P_{4}.&
\end{flalign*}
Notice that by \eqref{K5}, \eqref{(2.6)}, \eqref{(2.8)} and Young's inequality,
\begin{flalign*}
&\mathbb{E}(J_{3}+J_{4}) \leq 2\sqrt{M} \mathbb{E} \left( \int_{t_{i}^{n}}^{t} k_{1} \left(1+ \|u^{0}(s)\|^{2}\right) \|X^{h}(s)- \widetilde{Z}^{\varepsilon}(s)\|^{2}ds\right)^{1/2}&\\
&+ 2\varepsilon\hspace{.08cm} k_{3}^{2}  \log \log \frac{1}{\varepsilon}\hspace{.08cm} \mathbb{E} \int_{0}^{t_{i}^{n}} \|\widetilde{Z}^{\varepsilon}(s)\|^{2} |\dot{h}(s)|ds + \frac{1}{2} \hspace{.08cm}\mathbb{E} \int_{0}^{t_{i}^{n}}\|X^{h}(s)- \widetilde{Z}^{\varepsilon}(s)\|^{2}|\dot{h}(s)|ds&\\
&\leq 2\sqrt{M k_{1}} \left(N_{2}(T) (\varepsilon +\|\gamma\|^{2})\right)^{1/2} \left(\hat{N}_{2}(T,M) + N_{2}(T,M)\right)^{1/2} |t-t_{i}^{n}|^{1/2}&\\
&+ 2 \varepsilon \hspace{.08cm}k_{3}^{2}  \log \log \frac{1}{\varepsilon}\hspace{.08cm} N_{2}(T,M) \sqrt{M} |t_{i}^{n}|^{1/2} + \frac{1}{2} \left(\hat{N}_{2}(T,M) + N_{2}(T,M)\right)\sqrt{M} |t_{i}^{n}|^{1/2}&\\
&=: \widetilde{K}_{1}(\varepsilon,T,M)|t-t_{i}^n|^{1/2}+ \widetilde{K}_{2}(T, M)|t_{i}^n|^{1/2},&
\end{flalign*}
leading to,
\begin{eqnarray*}
&&\mathbb{E}\left(J_{1}+J_{2}+J_{3}+J_{4}+J_{5}\right) \leq \widetilde{K}_{1}(\varepsilon,T,M)|t-t_{i}^n|^{1/2}+ \widetilde{K}_{2}(\varepsilon, T, M)|t_{i}^n|\\
&&+ 2\sqrt{k_{0}} \left(\hat{N}_{2}(T,M) + N_{2}(T,M)\right) |t-t_{i}^{n}|^{1/2} \\
&&+ \frac{k_{1}}{2\log \log \frac{1}{\varepsilon}} \left(1+2\varepsilon \log \log \frac{1}{\varepsilon} N_{2}(T,M) + N_{2}(T)+ \|\gamma\|^{2}\right)|t_{i}^{n}|.
\end{eqnarray*}
Hence, probability $P_{1}$ is zero by an application of Chebyshev inequality and letting $n$ become arbitrary large since both $|t-t_{i}^{n}|$ and $|t_{i}^{n}|$ are less than $\frac{Ti}{2^{n}}$. For the second probability, we write the stochastic integral as a Riemman sum as follows.
\begin{flalign*}
&J_{7} = \frac{2}{\sqrt{2\log \log \frac{1}{\varepsilon}}}  \sum_{i=0}^{2^{n}-1} \|\widetilde{g}(t, \widetilde{Z}^{\varepsilon}(t))- \widetilde{g}(t, \widetilde{Z}^{\varepsilon}(t_{i}^{n}))\|_{L_{2}} \|\widetilde{Z}^{\varepsilon}(t)-\widetilde{Z}^{\varepsilon}(t_{i}^{n})\| \left(W(t_{i+1}^{n}) - W(t_{i}^{n})\right)&\\
&\leq \frac{2\sqrt{k_{3}}}{\sqrt{2\log \log \frac{1}{\varepsilon}}} \sum_{i=0}^{2^{n}-1} \|\widetilde{Z}^{\varepsilon}(t)-\widetilde{Z}^{\varepsilon}(t_{i}^{n})\|^{3/2} \left(W(t_{i+1}^{n}) - W(t_{i}^{n})\right)\leq \frac{2\sqrt{k_{3}}}{\sqrt{2\log \log \frac{1}{\varepsilon}}} \beta^{3/2} (2\eta) 2^{n}.&
\end{flalign*}
Therefore, with
\begin{equation*}
 \eta= \frac{\varphi \sqrt{2\log \log \frac{1}{\varepsilon}}}{2^{n+4} \sqrt{k_{3}} \beta^{3/2}},
 \end{equation*}
 the second probability, $P_{2}$, becomes zero. Furthermore, the third and fourth probabilities are used to obtain the exponential bound in \eqref{secondinequality2} by applying the same technique as in the proof of the exponential bound for \eqref{stochasticterm} in the previous step.
\end{proof}

 \section{The Strassen's compact LIL} \label{sec5}
Using the Freidlin-Wentzell inequality derived in the previous section, we prove the Strassen's compact LIL for the family $\{{Z^{\varepsilon}}\}_{\varepsilon\in (0,\varepsilon_{0})}$. We first provide a general definition of this type of LIL and for more background and similar results we recommend \cite{Ouahra, Caramellino, Eddahbi, Kouritzin,(38)}.
\begin{definition}(Strassen's Compact LIL).
A sequence $\{X_{j}\}_{j\geq 1}$ taking values in the space $(\mathcal{E},d)$ satisfies the Strassen's compact LIL if it is relatively compact and there exists a compact set  ${\mathcal K}$ in $\mathcal{E}$ such that the limit set of $\{X_{j}\}_{j\geq 1}$ is exactly ${\mathcal K}$.
\end{definition}
\begin{proof}[Proof of Theorem 2]
The family, $\{Z^{\varepsilon}\}_{\varepsilon \in (0,\varepsilon_{0})}$ is relatively compact in $\mathcal{C}([0,T];H)$, since by \eqref{vVnorm}, $\{Z^{\varepsilon}\}_{\varepsilon \in (0,\varepsilon_{0})}$ is  bounded in probability in space $W^{\alpha,2}(0,T;V)$, where $\alpha\in (1/2, 1)$ and Theorem 2.2 of \cite{(25)} may be applied to obtain that $\{Z^{\varepsilon} \}_{\varepsilon \in (0,\varepsilon_{0})}$ is tight in $\mathcal{C}([0,T];H)$.\\
\indent To ensure that the set $L$ given in Theorem \ref{TH2} is the set of limit points for $\{Z^{\varepsilon}\}_{\varepsilon \in (0,\varepsilon_{0})}$, we let $X^{h} $ be the unique solution of \eqref{controlled} with $I(X^{h} )\leq \frac{M}{2}$ and show that for every $\varepsilon>0$,
\begin{equation}\label{(4.7)}
P\left(\|Z^{\varepsilon} -X^{h}  \|_{\mbox{\scalebox{0.5}{$\infty$}}}\leq \varepsilon \hspace{.08cm}\text{ i.o. }\right)=1.
\end{equation}
For better presentation, we denote $\varepsilon:= \frac{1}{c^{j}}$ with $c>1$ and let $j\rightarrow \infty$. By the Strassen's compact LIL for Brownian paths (see \cite{(38)}), we have for the $Q$-Wiener process, $W$,
\begin{equation}\label{(4.8)}
P\left(\limsup_{j\rightarrow \infty} \left\|\frac{1}{\sqrt{2\log \log c^{j}}}W-h\right\|_{\mbox{\scalebox{0.5}{$\infty,0$}} } \leq \eta\right)=1,
\end{equation}
and by the Freidlin-Wentzell inequality \eqref{(3.2)}, for all $R>0, \rho>0$,
\begin{equation*}
P\left(\|Z^{\frac{1}{c^{j}}} -X^{h} \|_{\mbox{\scalebox{0.5}{$\infty$}}}>\rho, \left\|\frac{1}{\sqrt{2\log \log c^{j}}}W-h\right\|_{\mbox{\scalebox{0.5}{$\infty,0$}}} <\eta\right) \leq \exp\left(-2R\log \log c^{j}\right).
\end{equation*}
Since,
\begin{equation*}
 \exp\left(-2R\log \log c^{j}\right)\leq \frac{K}{j^{2R}},
\end{equation*}
we obtain \eqref{(4.7)} by applying the Borel-Cantelli lemma to the above Freidlin-Wentzell inequality and noting \eqref{(4.8)}.\\
\indent To verify that the set $L$ is the only limit set, we let $J:= \{g :\|g(t)-L\|_{\infty} \geq \varepsilon\}$ for some $\varepsilon>0$, implying that for $g\in J$, there exists $\delta>0$ such that $I(g)>\frac{M}{2}+\delta$. Since we have proved the MDP, then by its definition, we obtain for the closed set, $J$,
\begin{equation*}
\limsup_{j\rightarrow \infty} \frac{1}{2\log \log c^{j}} \log P\left(Z^{\frac{1}{c^j}}\in J\right) \leq -I(g)\leq -\left(\frac{M}{2}+\delta\right),
\end{equation*}
leading to,
\begin{equation*}
P\left(Z^{\frac{1}{c^{j}}}\in J\right) \leq \exp\left(-2\left(\frac{M}{2}+\delta\right)\log \log c^{j}\right)\leq \frac{1}{j^{M+2\delta}},
\end{equation*}
on which we may again apply the Borel-Cantelli lemma and obtain the convergence a.s. of $Z^{\frac{1}{c^{j}}}(t)$ to $L$ in $\mathcal{C}([0,T];H)$ as $j\rightarrow \infty$.
\end{proof}

\section{Azencott method and weak convergence approach}
Here we compare the two methods often applied in the literature to prove the LDP: the weak convergence approach, which was applied in \cite{Fatheddin2} to prove the MDP for \eqref{originalequation} and the Azencott method employed here to obtain the same result. For more examples of results on the LDP or MDP for SPDEs we refer the reader to \cite{Bessaih, Wang, Chueshov, Duan, Yang, Hu, Ren, Fatheddin3, Nualart3} for the weak convergence approach and to \cite{Nualart2, Peszat, Zhang} for the Azencott method. See Section 3 for more references for the Azencott method.   \\
\indent For the weak convergence approach, following \cite{BDM}, the space considered for the controlled function $h(\cdot)$ in $X^{h}$ given in \eqref{controlled} when the noise is Q-Wiener process is $\mathcal{P}_{M}$ defined below,
\begin{flalign*}
&\mathcal{P}_{2}:=  \left\{h:(0,T)\times \Omega\to H_{0}, \, \text{$(\mathcal{F}_{t})_{t\in [0,T]}$-predictable process}: \int_{0}^{T} |h(s)|_{0}^{2}ds <\infty\  \text{$P$-a.s.}\right\},&\\
&S_{M} := \left\{h \in L^{2}(0,T;H_{0}): \int_{0}^{T} |h(s)|_{0}^{2}ds \leq M\right\},&\\
&\mathcal{P}_{M} := \left\{h\in \mathcal{P}_{2}: h(\omega)\in S_{M}\hspace{.2cm} P\text{-a.s.}\right\},&
\end{flalign*}
where $M<\infty$ is a positive constant, where the space for other noise types such as Brownian sheet or infinite sequence of independent, standard  Brownian motions is defined similarly. Then based on this approach, we let the unique mild solution to the original SPDE be denoted by $\mathcal{G}^{\varepsilon}\left(x^{\varepsilon}, \sqrt{\varepsilon} W(\cdot)\right)$, the controlled PDE by $\mathcal{G}^{0}\left(x, \int_{0}^{\cdot}h(s)ds\right)$ and the stochastic controlled equation by $\mathcal{G}^{\varepsilon}\left(x^{\varepsilon}, \sqrt{\varepsilon} W(\cdot)+ \int_{0}^{\cdot} h_{\varepsilon}(s)ds\right)$ and prove the following two conditions to obtain the LDP with rate function,
\begin{align}\label{rate}
I(v) =
\frac{1}{2}\displaystyle \inf\left\{\int_{0}^{T} |h(s)|_{0}^{2}ds: h\in L^2([0,T];H_0)   \mbox{ with } v= \mathcal{G}^{0}\left(\int_{0}^{\cdot}h(s)ds\right)\right\},
\end{align}
where the infimum of the empty set is taken to be infinity.\\
\indent \emph{Conditions for the weak convergence approach:} Suppose $x, x^{\varepsilon}$ take values in Polish space, $\mathcal{E}_{0}$ and let $\mathcal{E}$ denote the Polish space in which the LDP is considered.
\begin{enumerate}
\item For a compact set $K$ in $\mathcal{E}_{0}$,
\begin{equation*}
  \Gamma_{M,K} = \left\{\mathcal{G}^{0}\left(x, \int_{0}^{\cdot} u(s)ds\right); u\in S_{M}, x\in K\right\},
  \end{equation*}
  is a compact subset of $\mathcal{E}$, where $M<\infty$ is a positive constant.
\item  For families, $\{h_{\varepsilon}\}_{\varepsilon>0} \subset \mathcal{P}_{M}$, $\{x^{\varepsilon}\}_{\varepsilon>0} \subset \mathcal{E}_{0}$, if $h_{\varepsilon}\xrightarrow{d} h$ and $x^{\varepsilon}\rightarrow x$ as $\varepsilon \rightarrow 0$, then
    \begin{equation*}
  \mathcal{G}^{\varepsilon} \left(x^{\varepsilon}, \sqrt{\varepsilon} W(\cdot) + \int_{0}^{\cdot} h_{\varepsilon}(s)ds\right) \xrightarrow{d}
  \mathcal{G}^{0}\left(x, \int_{0}^{\cdot} h(s)ds\right).
  \end{equation*}
  \end{enumerate}
  Thus, the conditions aim to prove that the family of controlled PDEs corresponding to each $h\in S_{M}$, is a compact set and the stochastic controlled equation converges in distribution to controlled PDE if $h_{\varepsilon}\xrightarrow{d} h$ and $x^{\varepsilon}\rightarrow x$ as $\varepsilon \rightarrow 0$. \\
  \indent On the other hand, as mentioned at the beginning of Section 3, the Azencott method requires the following two conditions. \\
  \indent \emph{Conditions for Azencott method:} Let $\{X_{1}^{\varepsilon}\}_{\varepsilon>0}$ be a process taking values in the Polish space $(\mathcal{E}_{1}, d_{1})$ that satisfies the LDP in $(\mathcal{E}_{1}, d_{1})$ with rate function, $\widetilde{I}: \mathcal{E}_{1} \rightarrow [0,\infty]$. Let $\{X_{2}^{\varepsilon}\}_{\varepsilon>0}$ be the process we seek to prove the LDP for and assume it takes values in the Polish space $(\mathcal{E}_{2}, d_{2})$.
  \begin{enumerate}
  \item The map $\Phi: \{g\in \mathcal{E}_{1}: \widetilde{I}(g)<\infty\} \rightarrow \mathcal{E}_{2}$ is continuous in the topology of $\mathcal{E}_{1}$ when restricted to the compact sets $\{\widetilde{I}\leq a\}_{0<a<\infty}$.
      \item For all $a>0, R>0, \rho>0$, there exists constants $\eta>0, \varepsilon_{0}>0$ such that for all $0<\varepsilon\leq \varepsilon_{0}, g\in \mathcal{E}_{1}$ and $\widetilde{I}(g)\leq a$ the Freidlin-Wentzell inequality given in \eqref{FreidlinWentzell} holds.
  \end{enumerate}
  In this setting, since the LDP is obtained with the rate function,
  \begin{equation*}
  I(f) = \inf \{\widetilde{I}(g): g\in \mathcal{E}_{1}, \Phi(g) = f\},
  \end{equation*}
  then map $\Phi$ refers to the map $h\mapsto X^{h}$ and based on the first condition of this method, one has to use the space in which LDP is known for $\{X_{1}^{\varepsilon}\}_{\varepsilon>0}$ to prove its continuity restricted to sets $\{\widetilde{I}\leq a\}_{0<a<\infty}$. For example, here we proved this continuity result in the uniform topology to match the space of $\{W/\sqrt{2\log \log (1/\varepsilon)}\}_{\varepsilon>0}$. For the first condition in the weak convergence approach, many authors such as those in \cite{Cheng, Hu, Li, Xu}, have proved the continuity of the map $h\mapsto X^{h}$ in the weak topology to achieve that the set $\Omega_{M,K}$ is compact, since the space $S_{M}$ in which $h$ takes values in is a compact set. There is also a slight variation in the space in which the function $h(\cdot)$ takes values in, where in the weak convergence approach it is in $\mathcal{P}_{M}$ and in Azencott method it is typically taken to be in the Cameron-Martin space associated with the noise used for the SPDE studied. However, each requires this function or its derivative to be in $L^{2}([0,T]; X)$, where $X$ is the corresponding space. Thus, the first condition in both methods aims to prove the same result with only some differences in spaces used to achieve it. \\
  \indent To compare the second condition in both methods, let us consider the inequality \eqref{(3.3)} used in this paper to prove the Friedlin-Wentzell inequality and translate it to the notation employed in the weak convergence approach as follows.
  \begin{flalign*}
  &P\left(\left\|\mathcal{G}^{\varepsilon}\left(x^{\varepsilon}, \sqrt{\varepsilon}W(\cdot)+ \int_{0}^{\cdot} h_{\varepsilon}(s)ds\right) - \mathcal{G}^{0}\left(x, \int_{0}^{\cdot} h(s)ds\right)\right\|_{\infty} \geq \rho, \left\|\frac{W}{\sqrt{2\log \log \frac{1}{\varepsilon}}} \right\|_{\infty, 0} <\eta\right)&\\
  & \leq \exp\left(-2R \log \log \frac{1}{\varepsilon}\right),&
  \end{flalign*}
  which must be true for any $R>0$, $\rho>0$ and $0<\varepsilon\leq \varepsilon_{0}$ for a fixed $\varepsilon_{0}>0$. Hence, as $\varepsilon\rightarrow 0$, $\mathcal{G}^{\varepsilon}\left(x^{\varepsilon}, \sqrt{\varepsilon}W(\cdot)+ \int_{0}^{\cdot} h_{\varepsilon}(s)ds\right)$ converges in probability to $\mathcal{G}^{0}\left(x, \int_{0}^{\cdot} h(s)ds\right)$ and thus in distribution, matching the second condition in the weak convergence approach.

  \section*{Acknowledgements}
  P. Fatheddin is thankful to Hannelore Lisei for the useful conversations on the exponential inequality needed for the Azencott method and for her helpful comments in preparing this article.


\begin{thebibliography}{38}
\bibitem{Ouahra}
M. Ait Ouahra, M. Mellouk (2005). Strassen's law of the iterated logarithm for stochastic Volterra equations and applications. \emph{Stoch.: An International J. Prob. Stoch. Proc.} Vol. 77, 191-203.
\bibitem{(1)}
A. Andresen, P. Imkeller, N. Perkowski (2013). \emph{Malliavin Calculus and Stochastic Analysis}. Springer, Proc. Math. Stat. 115-138.
\bibitem{Andrews}
L. Andrews, R. Philips (2005). \emph{Laser Beam Propagation through Random Media: Second Edition}. The Society of Photo-Optical Instrumentation Engineers (SPIE) Press: Bellingham, Washington, United States.
\bibitem{Azencott}
Azencott, R.: \emph{Grandes d$\acute{e}$viations et applications, Ecole d'$\acute{e}$t$\acute{e}$ de Probabilit$\acute{e}$ de Saint-Flour VIII,} Springer Lecture Notes in Mathematics, 1980.
\bibitem{(3)}
P. Baldi (1986). Large deviations and functional iterated logarithm law for diffusion processes. \emph{Probab. Th. Rel. Fields.}. Vol. 71, 435-453.
\bibitem{Bessaih}
H. Bessaih, A. Millet (2012). Large deviations and the zero viscosity limit for 2D stochastic Navier-Stokes equations with free boundary. \emph{SIAM J. Math. Anal.} Vol. 44, 1861-1893.
\bibitem{(6)}
N. Bingham (1986). Variants of the law of the iterated logarithm. \emph{Bull. London Math. Soc.} Vol. 18, 433-467.
\bibitem{(11)}
A. Budhiraja, P. Dupuis (2000). A varitiational representation for positive functionals of infinite-dimensional Brownian motion. \emph{Probab. Math. Stat.} Vol. 20, No. 1, 39--61.
\bibitem{BDM}
A. Budhiraja, P. Dupuis, V. Maroulas (2008). Large deviations for infinite dimensional stochastic dynamical systems. \emph{Ann. Probab.} Vol. 36, No. 4, 1390--1420.
\bibitem{Caramellino}
L. Caramellino (1998). Strassen's law of the iterated logarithm for diffusion processes for small time. \emph{Stoch. Proc. Appl.} Vol. 74, 1-9.
\bibitem{Cardon}
C. Cardon-Weber (1999). Large deviations for a Burger's type SPDE. \emph{Stoch. Proc. Appl.} Vol. 84, 53-70.
\bibitem{Cerrai}
S. Cerrai and M. R\"ockner (2004). Large deviations for stochastic reaction-diffusion systems with multiplicative noise and non-lipschitz reaction term. \emph{Ann. Probab.} Vol. 32, No. 1B, 1100-1139.
\bibitem{Chenal}
F. Chenal, A. Millet (1997). Uniform large deviations for parabolic SPDEs and applications. \emph{Stoch. Proc. Appl.} Vol. 72, 161-186.
\bibitem{Cheng}
L. Cheng, R. Li, R. Wang, N. Yao (2018). Moderate deviations for a stochastic wave equation in dimension three. \emph{Acta Appl. Math.} Vol. 158, 67--85.
\bibitem{Chow}
P. Chow and J. Menaldi (1990). Exponential estimates in exit probability for some diffusion processes in Hilbert space. \emph{Stoch. Stoch. Reports.} Vol. 29, 377-393.
\bibitem{Chueshov}
I. Chueshov, A. Millet (2010). Stochastic 2D hydrodynamical type systems: well-posedness and large deviations. \emph{Appl. Math. Optim.} Vol. 61, 379-420.
\bibitem{Deuschel}
J. Deuschel, D. Stroock (1989). \emph{Large Deviations}. \emph{Pure and Applied Mathematics}. Vol. 137, Academic Press, Inc. San Diego, USA.
\bibitem{Eddahbi1}
B. Djehiche and M. Eddahbi (1999). Large deviations for a stochastic Volterra-type equation in Besov-Orlicz space. \emph{Stoch. Proc. their Appl.} Vol. 81, 39-72.
\bibitem{Duan}
J. Duan, A. Millet (2009). Large deviations for the Boussinesq equations under random influences. \emph{Stoch. Proc. Appl.} Vol. 119, 2052-2081.
\bibitem{Eddahbi}
M. Eddahbi, M. N'Zi (2002). Strassen's local law for diffusion processes under strong topologies. \emph{ACTA Mathematica Vietnamica}. Vol. 27, No. 2, 151-163.
\bibitem{(23)}
P. Fatheddin (2020). The law of the iterated logarithm for a class of SPDEs. \emph{Stoch. Anal. Appl.} Vol. 39, No. 1, 113-135.
\bibitem{Fatheddin1}
P. Fatheddin (2022). \emph{Teaching and Research in Mathematics: A Guide with Applications to Industry}. CRC Press: Taylor and Francis Group. Boca Raton, USA.
\bibitem{(22)}
P. Fatheddin (2023). The law of the iterated logarithm for two-dimensional stochastic Navier-Stokes equation. \emph{J. Evolution Equations}. Vol. 23, No. 2, 28.
\bibitem{Fatheddin2}
P. Fatheddin, H. Lisei (2023). Moderate deviations for a stochastic Schr\"odinger equation with linear drift. \emph{Submitted}. arXiv: 2308.01488.
\bibitem{(24)}
P. Fatheddin, Z. Qiu (2020). Large deviations for nonlinear stochastic Schr\"odinger equation. \emph{Stoch. Anal. Appl.} Vol. 39, No. 3, 456-482.
\bibitem{Fatheddin3}
P. Fatheddin, J. Xiong (2015). Large deviation principle for some measure-valued process. \emph{Stoch. Proc. Appl.} Vol. 125, 970-993.
\bibitem{(25)}
F. Flandoli, D. Gatarek (1995). Martingale and stationary solutions for stochastic Navier-Stokes equations. \emph{Probab. Theory Relat. Fields.} Vol. 102, 367-391.
\bibitem{21}
E. Gautier (2005). Uniform large deviations for the nonlinear Schr\"odinger equation with multiplicative noise. \emph{Stoch. Proc. Anal.} Vol. 115, No. 12, 1904-1927.
\bibitem{22}
E. Gautier (2005). Large deviations and support results for nonlinear Schr\"odinger equations with additive noise and applications. \emph{ESAIM: Probab. Stat.} Vol. 9, 74-97.
\bibitem{23}
E. Gautier (2007). Stochastic nonlinear Schr\"odinger equations driven by a fractional noise well-posedness, large deviations and support. \emph{Elect. J. Probab.} Vol. 12, No. 29, 848-861.
\bibitem{Seidler}
E. Hausenblas and J. Seidler (2008). Stochastic convolutions driven by martingales: maximal inequalities and exponential integrability. \emph{Stoch. Anal. Appl.} Vol. 26, 98-119.
\bibitem{Hu}
S. Hu, R. Li, X. Wang (2020). Central limit theorem and moderate deviations for a class of semilinear SPDEs. \emph{Acta Math. Sci.} Vol. 40, 1477--1494.
\bibitem{Nualart3}
Y. Hu, D. Nualart, T. Zhang (2018). Large deviatons for stochastic heat equation with rough dependence in space. \emph{Bernoulli}. Vol. 24, 354-385.
\bibitem{Kouritzin}
M. Kouritzin, A. Heunis (1994). A law of the iterated logarithm for stochastic processes defined by differential equations with a small parameter. \emph{Ann. Probab.} Vol. 22, 659--679.
\bibitem{Li}
R. Li, X. Wang (2020). Central limit theorem and moderate deviations for a stochastic Cahn-Hilliard equation. \emph{Stoch. Dyn.} Vol. 20, 2050017.
\bibitem{Lisei1}
D. Keller, H. Lisei (2016). A stochastic nonlinear Schr\"odinger problem in variational formulation. \emph{NoDEA: Nonlinear Differ. Equ. Appl.} Vol. 23, No. 2, 21-27.
\bibitem{Nualart2}
A. Millet, D. Nualart, M. Sanz (1992). Large deviations for a class of anticipating stochastic differential equations. \emph{Ann. Probab.} Vol. 20, 1902-1931.
\bibitem{Nualart}
D. Nualart, C. Rovira (2000). Large deviations for stochastic Volterra equations. \emph{Bernoulli}. Vol. 6, No. 2, 339--355.
\bibitem{Pardoux}
E. Pardoux, A. Rascanu (2010). \emph{Stochastic Differential Equations, Backward SDEs, Partial Differential Equations}. Stochastic Modeling and Applied Probability, Vol. 69, Springer, New York.
\bibitem{Peszat}
S. Peszat (1994). Large deviation principle for stochastic evolution equations. \emph{Probab. Theory Relat. Fields.} Vol. 98, 113-136.
\bibitem{Priouret}
Priouret, P.: \emph{Remarques sur les Petites Perturbations de Syst$\acute{e}$mes Dynamiques}, Springer S$\acute{e}$minaire de Probabilit$\acute{e}$s XVI Lecture Notes in Mathematics, 1982.
\bibitem{Ren}
J. Ren, X. Zhang (2008). Freidlin-Wentzell's large deviations for stochastic evoluation equations. \emph{J. Functional Anal.} Vol. 254, 3148-3172.
\bibitem{(38)}
V. Strassen (1964). An invariance principle for the law of the iterated logarithm. \emph{Z. Wahrschein-lichkeitstheorie}. Vol. 3, 211--226.
\bibitem{(39)}
J. W. Strohbehn (1978). \emph{Laser Beam Propagation in the Atmosphere.} Topics in Applied Physics. Vol. 25, Springer, New York.
\bibitem{Wang}
R. Wang, J. Zhai, T. Zhang (2015). A moderate deviation principle for 2D stochastic Navier-Stokes equations. \emph{J. Diff. Eq.} Vol. 258, 3363-3390.
\bibitem{Wu}
L. Wu (1994). Large deviations, moderate deviations and LIL for empricial processes. \emph{Ann. Probab.} Vol. 22, No. 1, 17--27.
\bibitem{Xu}
T. Xu, T. Zhang (2009). White noise driven SPDEs with reflection: existence, uniqueness and large deviation principles. \emph{Stoch. Proc. Appl.} Vol. 119, 3453--3470.
\bibitem{Yang}
J. Yang, J. Zhai (2017). Asymptotics of stochastic 2D hydrodynamical type systems in unbounded domains. \emph{Infinite Dimen. Analy., Quantum Probab. Related Topics}. Vol. 20, 1750017.
\bibitem{Zhang}
T. Zhang (2007). Large deviations for stochastic nonlinear beam equations. \emph{J. Functional Anal.} Vol. 248, 175-201.
\end{thebibliography}
\end{document}